\documentclass[a4paper,11pt,DIV=calc]{scrartcl}

\usepackage[utf8]{inputenc}
\usepackage[T1]{fontenc}
\usepackage{lmodern}

\usepackage{amssymb}
\usepackage{amsmath}
\usepackage{amsfonts}
\usepackage{mathtools}
\usepackage{bm}
\usepackage{authblk}
\usepackage{comment}

\usepackage{xcolor}

\usepackage{tikz}
\usepackage{pgfplots}
\usetikzlibrary{patterns,intersections,external}
\usepgfplotslibrary{fillbetween}
\pgfplotsset{compat=1.14}
\usepackage{hyperref}
\hypersetup{hidelinks}

\usepackage[activate = {true}, final, tracking = true, kerning = true, spacing =
true, factor = 1100, stretch = 17, shrink = 17, letterspace=30]{microtype}
\microtypecontext{spacing=nonfrench}

\usepackage[shortlabels]{enumitem}

\setlist[description]{font=\normalfont\bfseries\itshape,leftmargin=*}

\usepackage{booktabs}

\usepackage{array}

\usepackage{subcaption}

\usepackage{amsthm}

\usepackage{float}

\usepackage{todonotes}

\usepackage[doi=false,url=false,giveninits=true,hyperref,maxbibnames=10,backend=biber]{biblatex}
\bibliography{main}

\newtheorem{theorem}{Theorem}
\newtheorem{lemma}[theorem]{Lemma}
\newtheorem{proposition}[theorem]{Proposition}
\newtheorem{corollary}[theorem]{Corollary}

\theoremstyle{definition}

\newtheorem*{definition}{Definition}

\DeclareMathOperator{\pbinom}{Bin}

\DeclareMathOperator{\pbeta}{Beta}
\DeclareMathOperator{\pgamma}{Gamma}

\DeclareMathOperator{\sign}{Sign}
\DeclareMathOperator{\betafun}{B}

\DeclareMathOperator{\mode}{mode}
\DeclareMathOperator{\antimode}{anti-mode}

\newcommand{\R}{\mathbb{R}}
\newcommand{\N}{\mathbb{N}}

\newcommand{\stleq}{\mathrel{\leq_{\mathrm{st}}}}
\newcommand{\cleq}{\mathrel{\leq_c}}
\newcommand{\IFRAleq}{\mathrel{\leq_\ast}}

\DeclareSymbolFont{ttfont}{OT1}{cmtt}{m}{n}
\DeclareMathSymbol{\splus}{\mathord}{ttfont}{`+}
\DeclareMathSymbol{\sminus}{\mathord}{ttfont}{`-}
\DeclareMathSymbol{\snull}{\mathord}{ttfont}{`0}

\DeclareMathOperator{\srev}{rev}

\title{Convex transform order of Beta distributions with some consequences\thanks{This work was partially supported by the Centre for Mathematics of the University of Coimbra - UIDB/00324/2020, funded by the Portuguese Government through FCT/MCTES.}}

\author[1]{Idir Arab}
\author[1]{Paulo Eduardo Oliveira}
\author[2]{Tilo Wiklund}

\affil[1]{CMUC, Department of Mathematics, University of Coimbra, Coimbra, 3001-501, Portugal}
\affil[2]{Department of Mathematics, Uppsala University, Uppsala, 751 06, Sweden}

\newcommand{\correction}[2]{#2}

\date{\vspace{-\baselineskip}}

\begin{document}

\maketitle

\begin{abstract}
The convex transform order is one way to make precise comparison between the
  skewness of probability distributions on the real line. We establish a simple
  and complete characterisation of when one Beta distribution is smaller than
  another according to the convex transform order. As an application, we derive
  monotonicity properties for the probability of Beta distributed \correction{}{random} variables
  exceeding the mean or mode of their distribution. Moreover\correction{}{,} we
  obtain a simple alternative proof of the mode-median-mean inequality for
  unimodal distributions that are skewed in a sense made precise by the convex
  transform order. This new proof also gives an analogous inequality for the
  anti-mode of distributions that have a unique anti-mode. Such inequalities for
  Beta distributions follow as special cases. Finally, some consequences for the
  values of distribution functions of Binomial distributions near to their means
  are mentioned.
\end{abstract}

\section{Introduction}
\label{sec:intro}

How to order probability distributions according to criteria that have
consequences with probabilistic interpretations is a common question in probability
theory. Naturally, there will exist many  order relations, each one
highlighting a particular aspect of the distributions. Classical examples are
given by orderings that capture size and dispersion. In reliability theory, some
ordering criteria are of interest when dealing with ageing problems. These help
decide, for example, which lifetime distributions exhibit faster ageing. An
account of different orderings, their properties, and basic relationships may be
found in the monographs of 
\textcite{MO07} or
\textcite{SS07}.

In this paper we shall be interested primarily in two such orderings. In the literature they are known as the convex transform order and the star-shape transform order.
These orders are defined by the convexity or star-shapedness of a certain
mapping that transforms one distribution into another. The convex transform
order was introduced by 
\textcite{VZ-64} with the aim of comparing skewness properties of
distributions. 
\textcite{Oja81} suggests that any measure of skewness should be compatible with the
convex transform order, and that many such measures indeed are. Hence, this
ordering gives a convenient formalisation of what it means to compare
distributions according to skewness.

With respect to the ageing interpretation, the convex transform order may be
seen as identifying ageing rates in a way that works also when lifetimes did not
start simultaneously. In this context, the star-shape order requires the same starting
point for the distributions under comparison, as described in
\textcite{ASOK}.

Establishing that one distribution is smaller than another is often difficult
and tends to rely on being able to control the number of crossing points between suitable transformations of distribution functions. This led \textcite{Bel-et-al-2013} to propose a criterion for deciding about the star-shape order between two distributions with the same support based on a quotient of suitably scaled densities. This same quotient was used in \textcite{AHO20a} to derive some negative results about the comparability with respect to the convex ordering. Using a different approach, depending on the number of modes of an appropriate transformation of the inverse distribution functions, \textcite{Arriaza-et-al-2019} gave sufficient conditions for the convex ordering. By relating it to an idea of \textcite{Parzen1979}, \textcite{Alzaid1989} describe a connection of the ordering to, roughly speaking, the tail behaviours of the two distributions. More recently, based on the analysis of the sign variation of affine transformations of distribution functions, 
\textcite{AO19,AO18} and
\textcite{AHO20} proved explicit ordering relationships within the Gamma and Weibull families.

The family of Beta distributions is a two-parameter family of distributions
supported on the unit interval. It appears , for example, in
the study of order statistics and in Bayesian statistics as a conjugate prior
for a variety of distributions arising from Bernoulli trials.

The main contribution of this paper is to characterise when one Beta
distribution is smaller than another according to the convex- and star-shaped-transform orders. This characterisation implies various monotonicity properties
for the probabilities of Beta distributed random variables exceeding the mean or
mode of their distribution. Using this allows one to derive, in some cases,
simple bounds for such probabilities. These bounds differ from concentration
inequalities such as Markov's inequality or Hoeffding's inequality in that they
control the probability of exceeding, without necessarily significantly
\emph{deviating} from, the mean or the mode.

A well-known connection between Beta and the Binomial distributions allows us to
translate these results into similar monotonicity properties for the family of
Binomial distributions. The question of controlling the probability of a binomially distributed quantity exceeding its mean has received
attention in the context of studying properties of randomised algorithms, see
for example 
\textcite{Karppa}, 
\textcite{Bicchetti}, or 
\textcite{Mitzenmacher}. The question also appears when dealing with specific
aspects in machine learning problems, such as in 
\textcite{doerr18}, 
\textcite{Greenberg}, with sequels in 
\textcite{Pelikis:lower} and 
\textcite{Pelikis:tail}. See 
\textcite{Cortes} for more general questions. Such an inequality for the Binomial
random variables was also used by 
\textcite{Wik18} when studying the amount of information lost when resampling.

These properties also allow one to compare the relative location of the mode,
median, and mean of certain distributions that are skewed in a sense made
precise by the convex transform order. Such mode-median-mean inequalities are a
classical subject in probability theory. While our condition for these
inequalities to hold has previously been suggested by 
\textcite{VZ-79}, our proof appears novel. The proof also allows us to establish a
similar inequality for absolutely continuous distributions with unique
anti-modes, meaning distributions that have densities with a unique minimizer. For an
account of the field we refer the interested reader to 
\textcite{VZ-79} or, for more recent references, to 
\textcite{Aba05} or
\textcite{Zheng17}.

This paper is organised as follows. Section~\ref{sec:prel} contains a review of
important concepts and definitions. The main results, characterising the order
relationships within the Beta family, are presented in Section~\ref{sec:main}.
Consequences are discussed in Section~\ref{sec:applications}, while proofs of
the main results are presented in Section~\ref{sec:proofs}. Some auxiliary
results concerning the main tools of analysis are given later in
Appendix~\ref{sec:overview}.

\section{Preliminaries}
\label{sec:prel}

In this section we present the basic notions necessary for understanding the
main contributions of the paper.

Let us first recall the classical notion of convexity on the real numbers.
\begin{definition}[Convexity]
  A real valued function \(f \colon I \mapsto \R\) on an interval \(I\) is said
  to be \emph{convex} if for every \(x, y \in I\) and \(\alpha \in [0, 1]\) we
  have \(f(\alpha x + (1-\alpha) y) \leq \alpha f(x) + (1-\alpha) f(y)\).
\end{definition}
We will also need the somewhat less well-known notion of star-shapedness of a
function on the real numbers.
\begin{definition}[Star-shapedness]
  A function \(f \colon [0,a] \to \R\), for some \(a \in (0, \infty]\), is said to be
  \emph{star-shaped} if for every \(0 \leq \alpha \leq 1\), we have \(f(\alpha
  x) \leq \alpha f(x)\).
\end{definition}
Star-shapedness can be defined on general intervals with respect to
an arbitrary reference point. For our purposes it suffices to consider functions
on the non-negative half-line that are star-shaped with the origin as reference point.

A convex \(f \colon I \to \R\) on an initial segment of the
non-negative half-line that satisfies \(f(0) \leq 0\) is star-shaped. Moreover,
\(f\) is star-shaped if and only if \(f(x)/x\) is increasing in \(x \in I\). We
refer the reader to 
\textcite{BMP69} for some more general properties and relations between these types
of functions.

Our main concern in this paper is to establish certain orderings of the family
of Beta distributions that are defined in \([0,1]\).
\begin{definition}[Beta distribution]
  The Beta distribution \(\pbeta(a, b)\) with parameters \(a, b > 0\) is a
  distribution supported on the unit interval and defined by the density given
  for \(x \in (0, 1)\) by
  \begin{equation}\label{eq:dbetas}
    \frac{x^{a-1}(1-x)^{b-1}}{\betafun(a, b)}
    \quad
    \text{where}
    \quad
    \betafun(a, b) = \int_{0}^{1}y^{a-1}(1-y)^{b-1}\,dy.
  \end{equation}
\end{definition}

We will consider two orderings determined by the
convexity or star-shapedness of a certain mapping. Of primary interest is the following order
due to 
\textcite{VZ-64}. In order to avoid working with generalised inverses, we restrict
ourselves to distributions supported on intervals.
\begin{definition}[Convex Transform Order \(\boldsymbol\cleq\)]\label{def:corder}
  Let \(P\) and \(Q\) be two probability distributions on the real line
  supported by the intervals \(I\) and \(J\) that have strictly increasing
  distribution functions \(F \colon I \to [0, 1]\) and \(G \colon J \to [0,
  1]\), respectively. We say that \(P \cleq Q\) or, equivalently, \(F \cleq G\),
  if the mapping \(x \mapsto G^{-1}(F(x))\) is convex. Moreover, if \(X \sim P\)
  and \(Y \sim Q\), we will also write \(X \cleq Y\) when \(P \cleq Q\).
\end{definition}

If \(X \sim P\) and \(Y \sim Q\) then
both \(X \cleq Y\) and \(Y \cleq X\) if and only if there exist some \(a > 0\)
and \(b \in \R\) such that \(X\) has the same distribution as \(aY + b\). In
other words, the convex transform order is invariant under orientation-preserving affine transforms.

Although it is popular in reliability theory, the convex
transform order was first introduced by 
\textcite{VZ-64} to compare the shape of distributions with respect to skewness
properties. The idea is roughly as follows. Let \(X\) and \(Y\) be random
variables having, say, absolutely continuous distributions given by distribution
functions \(F\) and \(G\), respectively. Then \(G^{-1}(F(X))\) has the same law
as \(Y\). Convexity of \(x \mapsto G^{-1}(F(x))\) implies that the transformed
distribution tends to be spread out in the right tail while being compressed in
the left tail. In other words, \(Y\) will have a distribution more skewed to the
right. Indeed, if \(\psi\) is an increasing function then \(X \cleq \psi(X)\) if
and only if \(\psi\) is convex.

In the reliability literature the convex transform ordering is known as the
\emph{increasing failure rate} (\textsc{ifr}) order. Indeed, assuming that \(F\)
and \(G\) are absolutely continuous distribution functions with derivatives
\(f\) and \(g\) and failure rates \(r_{F} = f/(1-F)\) and \(r_{G} = g/(1-G)\)
then \(F \cleq G\) is equivalent to
\begin{equation*}
  \frac{f(F^{-1}(u))}{g(G^{-1}(u))} = \frac{r_{F}(F^{-1}(u))}{r_{G}(G^{-1}(u))}
\end{equation*}
being increasing in \(u \in [0, 1]\).

The second order of interest is defined analogously to the convex transform
order, but now with respect to star-shapedness.

\begin{definition}[Star-shaped order \(\boldsymbol\IFRAleq\)]\label{def:starorder}
  Let \(P\) and \(Q\) be two probability distributions on the real line
  supported by the intervals \(I = [0, a]\) and \(J = [0, b]\), for some \(a, b
  > 0\), and which have strictly increasing distributions functions \(F \colon I \to [0,
  1]\) and \(G \colon J \to [0, 1]\), respectively. We say that \(P \IFRAleq Q\)
  or, equivalently, \(F \IFRAleq G\), if the mapping \(x \mapsto G^{-1}(F(x))\)
  is star-shaped. Moreover, if \(X \sim P\) and \(Y \sim Q\), we will also write
  \(X \IFRAleq Y\) when \(P \IFRAleq Q\).
\end{definition}
If \(X \sim P\) and \(Y \sim Q\) for appropriate \(P\) and \(Q\) then \(X
\IFRAleq Y\) and \(Y \IFRAleq X\) if and only if there exists an \(a > 0\) such
that \(X\) has the same distribution as \(aY\).

The star transform order can be interpreted in terms of the average failure
rate. It is therefore sometimes known as the \emph{increasing failure rate on
  average} (\textsc{ifra}) order. In fact, \(F \IFRAleq G\) is equivalent to
\(G^{-1}(u)/F^{-1}(u)\) being increasing in \(u \in [0, 1]\). Moreover,
\begin{equation*}
  \frac{G^{-1}(x)}{F^{-1}(x)} = \frac{\overline{r}_{F}(F^{-1}(u))}{\overline{r}_{G}(G^{-1}(u))},
\end{equation*}
where \(\overline{r}_{F}(x)\) and \(\overline{r}_{G}(x)\) are known as the
failure rates on average of \(F\) and \(G\), respectively, and are defined by
\(\overline{r}_{F}(x) = -\ln(1-F(x))/x\) and \(\overline{r}_{G}(x) =
-\ln(1-G(x))/x\).

The star-shaped order is strictly weaker than the convex transform order
for distributions having support with a lower end-point at \(0\), such as the
Beta distributions. That being said, it is of some independent interest as well as being useful as an intermediate order when establishing ordering according to the convex transform order.

The stochastic dominance order is also known as first stochastic dominance
(\textsc{fsd}) in reliability theory, and captures the notion of one
distribution attaining larger values than the other. It is generally easier to
verify than the convex transform order or star-shaped order and will serve here
primarily to establish necessity of the sufficient conditions for convex
transform ordering between two Beta distributions.
\begin{definition}[Stochastic dominance \(\stleq\)]
  Let \(P\) and \(Q\) be two probability distributions on the real line with
  distributions functions \(F \colon \R \to [0, 1]\) and \(G \colon \R \to [0, 1]\),
  respectively. We say that \(P \stleq Q\) or, equivalently, \(F \stleq G\), if
  \(F(x) \geq G(x)\), for all \(x \in \R\). Moreover, if \(X \sim P\) and \(Y \sim Q\),
  we will also write \(X \stleq Y\) when \(P \stleq Q\).
\end{definition}

\section{Main results}
\label{sec:main}

The main results of this paper describe the stochastic dominance-, star-shape
transform-, and convex transform-order relationships within the family of
Beta distributions. The proofs are postponed until Section~\ref{sec:proofs}.

The following stochastic dominance order relationships within the family of Beta distributions
are known, and can be found in, for example, the appendix of \textcite{Lisek1978}. 
\begin{theorem}\label{thm:betast}
  Let \(X \sim \pbeta(a, b)\) and \(Y \sim \pbeta(a', b')\), then \(Y \stleq X\)
  if and only if \(a \geq a'\) and \(b \leq b'\).
\end{theorem}
The star-shape ordering relationships within the family of Beta distributions
have been addressed previously by \textcite{JKP06} (see Example~4), but only
for the case of integer valued parameters that satisfy certain conditions. Here
we extend this to a complete classification.
\begin{theorem}\label{thm:betaifra}
  Let \(X \sim \pbeta(a, b)\) and \(Y \sim \pbeta(a', b')\), then \(X \IFRAleq
  Y\) if and only if \(a \geq a'\) and \(b \leq b'\).
\end{theorem}

Two
Beta distributions turn out to be ordered according to the convex transform order if and
only if they are ordered according to the star-shaped order.
\begin{theorem}\label{thm:mainifr}
  Let \(X \sim \pbeta(a, b)\) and \(Y \sim \pbeta(a', b')\), then \(X \cleq Y\)
  if and only if \(a \geq a'\) and \(b \leq b'\).
\end{theorem}

\section {Some consequences of the main results}
\label{sec:applications}

A first simple result follows from the invariance of the convex ordering under
affine transformations. Recall that the family of Gamma distribution with
parameters \(\alpha,\theta>0\), denoted \(\pgamma(\alpha, \theta)\), is defined
by the density functions given for \(x > 0\) by
\begin{equation*}
  \frac{x^{\alpha-1}e^{-x/\theta}}{\theta^\alpha\Gamma(\alpha)}
  \quad\text{where}\quad
  \Gamma(\alpha) = \int_{0}^{\infty}y^{\alpha-1}e^{-y}\,dy.
\end{equation*}
Taking \(X_{b} \sim \pbeta(a, b)\) for some \(a > 0\) fixed and letting \(b\)
tend to \(+\infty\), the distributions of \(bX_{b}\) converges weakly to
\(\pgamma(a, 1)\). The following proposition is therefore an immediate
consequence of the transitivity of the transform orders and
Theorems~\ref{thm:betaifra}~and~\ref{thm:mainifr}.
\begin{proposition}\label{prop:gamma}
  Let \(X \sim \pbeta(a, b)\) and \(Y \sim \pgamma(a, \theta)\) for \(a, b,
  \theta > 0\), then \(X \IFRAleq Y\) and \(X \cleq Y\).
\end{proposition}

We considered the beta distribution defined with support \([0,1]\), therefore
the inverse of the distribution function defines another class of distributions
dubbed the complementary beta distributions, studied by \textcite{Jones2002}. The
convex transform order between two complementary beta distributions is then
expressed through the convexity of \(G(F^{-1}(x))\), where \(F\) and \(G\) are
beta distribution functions. This convexity is equivalent to the likelihood
ratio order (see for example chapter 1.C of \textcite{SS07}) between the beta
distributions. Hence the convex transform order between the complementary beta
distributions translates to the likelihood ratio order between beta
distributions and vice versa. Consequently, a characterisation of when two
complementary beta distributions are ordered according to the likelihood ratio
order follows immediately from Theorem~\ref{thm:mainifr}. We thank the anonymous
reviewer who pointed out this connection.

\subsection{Probabilities of exceedance}

It was noted already by 
\textcite{VZ-64} that the probabilities of random variables being greater than (or
smaller than) their expected values is monotone with respect to convex transform
ordering of their distributions. As we will see, this is a consequence of
Jensen's inequality. The idea generalises directly to any functional that
satisfies a Jensen-type inequality.

\begin{theorem}\label{thm:genjensen}
  For any interval \(I\), measurable function \(h \colon
    I \to \R\) and \(X \sim P\) with \(P\) supported in \(I\) denote the
    distribution of \(h(X)\) by \(P_{h}\).

  Let \(\mathcal{F}\) be a set of continuous probability distributions
  on intervals in \(\R\) and \(T \colon \mathcal{F} \to
  \R\) a functional satisfying for all \(P \in \mathcal{F}\) and \(h\) convex
  and increasing with \(P_{h} \in \mathcal{F}\) that \(h(T(P)) \leq T(P_{h})\).

  Then if \(X \sim P\) and \(Y \sim Q\) with distributions \(P, Q \in
  \mathcal{F}\) such that \(X \cleq Y\) it holds that \(\mathbb{P}(X \geq T(P)) \geq
  \mathbb{P}(Y \geq T(Q))\).

  If \(T\) satisfies instead \(h(T(P)) \geq T(P_{h})\) then, under the same
  assumptions on \(X\) and \(Y\), the conclusion becomes \(\mathbb{P}(X \geq T(P))
  \leq \mathbb{P}(Y \geq T(Q))\).
\end{theorem}
\begin{proof}
  Assume \(T\) satisfies the first inequality, \(h(T(P)) \leq T(P_{h})\). Let
  \(F\) and \(G\) be the distribution functions of \(X\) and \(Y\),
  respectively, and \(h(x) = G^{-1}(F(x))\). Since both \(F\) and \(G\) are increasing, so is \(h\). The
  assumption \(X \cleq Y\) implies \(h\) is also convex so that
  \(G^{-1}(F(T(P))) = h(T(P)) \leq T(P_{h}) = T(Q)\). Since \(G\) is increasing
  it follows that \(F(T(P)) \leq G(T(Q))\).

  The second statement, for \(T\) satisfying \(h(T(P))
    \leq T(P_{h})\), follows by reproducing the same argument with the
  inequality reversed.
\end{proof}

The standard Jensen inequality implies that we may take as \(T\)
in Theorem~\ref{thm:genjensen} the expectation operator
  \(T(P) = \mathbb{E}(X)\) for \(X \sim P\). Hence we recover the result of van~Zwet mentioned above.
\begin{corollary}\label{corr:mainexpectation}
  Let \(X\) and \(Y\) be two random variables such that \(X \cleq Y\). Then
  \(\mathbb{P}(X \geq \mathbb{E}(X)) \geq \mathbb{P}(Y \geq \mathbb{E}(Y))\).
\end{corollary}
Together with Theorem~\ref{thm:mainifr} this corollary now gives the following
monotonicity properties of Beta distributed random variables exceeding their
expectation.
\begin{corollary}\label{cor:exceeding_means}
  For each \(a, b > 0\) let \(X_{a, b} \sim \pbeta(a, b)\). Then \((a, b)
  \mapsto \mathbb{P}(X_{a, b} \geq \mathbb{E}(X_{a, b}))\) is increasing in \(a\) and
  decreasing in \(b\).
\end{corollary}

This provides immediate bounds for the probabilities of Beta distributed random
variables exceeding their expectation.
\begin{corollary}\label{cor:beta_over_mean}
  Let \(X_{a, b} \sim \pbeta(a, b)\), where \(a, b \geq 1\). Then
  \begin{equation*}
    e^{-1}
    <
    \left(\frac{b}{1+b}\right)^{b}
    \leq
    \mathbb{P}(X_{a, b} \geq \mathbb{E}(X_{a, b}))
    \leq
    1 -\left(\frac{a}{1+a}\right)^{a}
    <
    1-e^{-1}.
  \end{equation*}
\end{corollary}
\begin{proof}
  Compute \(\mathbb{P}(X_{a, b} \geq \mathbb{E}(X_{a, b}))\) for \(a = 1\) or \(b = 1\),
  use the monotonicity given in Corollary~\ref{cor:exceeding_means}, and,
  finally allow \(a, b \to +\infty\) to find both numerical bounds.
\end{proof}

Using Theorem~\ref{thm:genjensen} we may prove similar monotonicity properties
for the probabilities of exceeding modes or anti-modes. Recall that an absolutely
continuous distribution is unimodal if it has a
continuous density with a unique maximizer and
uniantimodal if it has a continuous density with a
unique minimizer.
\begin{corollary}\label{corr:modejensen}
  Let \(X \sim P\) and \(Y \sim Q\) be two real valued random variables with
  absolutely continuous distributions \(P\) and \(Q\) supported on some
  intervals \(I\) and \(J\) and such that \(X \cleq Y\).

  If \(P\) and \(Q\) are unimodal with modes \(\mode(X)\) and \(\mode(Y)\),
  respectively, then \(\mathbb{P}(X \geq \mode(X)) \leq \mathbb{P}(Y \geq \mode(Y))\).

  If \(P\) and \(Q\) are uniantimodal with anti-modes \(\antimode(X)\) and
  \(\antimode(Y)\), respectively, then \(\mathbb{P}(X \geq \antimode(X)) \geq \mathbb{P}(Y
  \geq \antimode(Y))\).
\end{corollary}
\begin{proof}
  We prove only the result about modes, as the statement about anti-modes
  follows analogously.

  Define \(\mathcal{F}\) as the set of absolutely continuous unimodal
  distributions supported in some interval in \(\R\) and \(T \colon \mathcal{F}
  \to \R\) the functional defined by \(T(P)\) being equal to the unique mode of
  \(P\), for every \(P \in \mathcal{F}\). By Theorem~\ref{thm:genjensen} it
  suffices to prove that \(T\) satisfies \(h(T(P)) \leq T(P_{h})\), for every
  \(P \in \mathcal{F}\), and \(h\) convex and increasing such that
  \(P_h\in\mathcal{F}\). For this purpose, choose \(f\) to be a continuous and
  unimodal version of the density of \(P\), and denote, for notational
  simplicity, the unique mode by \(m\). It is immediate that \(g(x) =
  f(h^{-1}(x))/h'(h^{-1}(x))\) is a density for \(P_{h}\). Since \(P_{h}\) has
  some continuous density with a unique mode and \(h\) is increasing and convex,
  \(g\) must be such a density. Denote the mode \(T(P_{h})\) by \(m'\).

  Since \(m\) is a mode of \(P\) it follows that \(f(m)\geq f(h^{-1}(m'))\) and,
  by the unimodality of \(P_{h}\), it follows that
  \begin{equation*}
    \frac{f(h^{-1}(m'))}{h'(h^{-1}(m'))} = g(m') \geq g(h(m)) =  \frac{f(m)}{h'(m)}.
  \end{equation*}
  Consequently \(h'(h^{-1}(m')) \leq h'(m)\), which in turn implies that
  \(m'\leq h(m)\), since \(h'\) and \(h\) are both increasing. The conclusion
  now follows immediately from Theorem~\ref{thm:genjensen}.
\end{proof}

Similarly to Corollary~\ref{cor:exceeding_means}, the previous result implies
monotonicity properties for the probability of exceeding the mode or anti-mode
for Beta distributions. For this to work we must restrict ourselves to parameters \(a\)
and \(b\) such that \(\pbeta(a, b)\) actually has a unique mode or anti-mode. This happens  when \(a, b > 1\) or \(a, b < 1\), respectively. In either case
the mode or anti-mode is \((a-1)/(a+b-2)\).
\begin{corollary}\label{cor:mainmode}
  For \(a, b > 0\) let \(X_{a,b} \sim \pbeta(a, b)\).

  If \(a, b > 1\) let \(\mode(X_{a,b})\) be the mode of \(\pbeta(a, b)\), then the mapping \((a,
  b) \mapsto \mathbb{P}(X_{a,b} > \mode(X_{a,b}))\) is decreasing in \(a\) and
  increasing in \(b\).

  If \(a, b<1\) let \(\antimode(X_{a,b})\) be the anti-mode of \(\pbeta(a, b)\), then the mapping
  \((a, b) \mapsto \mathbb{P}(X_{a,b}>\antimode(X_{a,b}))\) is increasing in \(a\)
  and decreasing in \(b\).
\end{corollary}
Recall that \(B \sim \pbinom(n, p)\) if \(\mathbb{P}(B = k) =
\binom{n}{k}p^{k}(1-p)^{n-k}\)for \(n = 1, 2, \dotsc\) and \(k \in \{1, \dotsc, n\}\). Using a link between the Beta and the Binomial distributions
allows us to prove some monotonicity properties for the probabilities that a
Binomial variable exceeds certain values close to its mean. As noted in the
Introduction, the quantity \(\mathbb{P}(B_{n,p} \leq np)\), where \(B_{n,p} \sim
\pbinom(n,p)\) has garnered some interest recently. The mapping \(p \mapsto
\mathbb{P}(B_{n,p} \leq np)\) is not monotone even when restricting to \(p = 0, 1/n,
\ldots, (n-1)/n, 1\), where \(np\) is an integer. Using our results we prove
that slightly changing \(np\) renders monotonicity.
\begin{corollary}\label{cor:binomial}
  For \(n = 2, 3, \dotsc\) and for each \(p \in [0, 1]\) let \(B_{n,p} \sim
  \pbinom(n,p)\). The mapping \(p \mapsto \mathbb{P}(B_{n,p} > np-p)\) is increasing for
  \(p = 1/(n-1), \ldots, (n-2)/(n-1)\), and the mapping \(p \mapsto \mathbb{P}(B_{n, p} >
  np-(1-p))\) is decreasing for \(p = 1/(n+1), \ldots, n/(n+1)\).
\end{corollary}
\begin{proof}
  For each \(a, b > 0\) let \(X_{a,b} \sim \pbeta(a,b)\). It is well-known that
  \(\mathbb{P}(X_{k+1,n-k} \geq p) = \mathbb{P}(B_{n,p} \leq k)\), for \(k = 0, \ldots,
  n\). The equality can for example be established by repeated integration by
  parts. As the distribution of \(X_{k+1,n-k}\) has mean \((k+1)/(n+1)\) and
  mode \(k/(n-1)\), it follows from
  Corollaries~\ref{cor:exceeding_means}~and~\ref{cor:mainmode}, that \(k \mapsto
  \mathbb{P}(B_{n,\frac{k+1}{n+1}} \geq k)\) is decreasing and \(k \mapsto
  \mathbb{P}(B_{n,\frac{k}{n-1}}\geq k)\) is increasing. Reparameterising in terms of
  \(p\) yields \(k=np+p-1\) and \(k=np-p\), so the result follows.
\end{proof}

\subsection{(Anti)mode-median-mean inequalities}

If \(X_{a, b} \sim \pbeta(a, b)\) then the random variable \(1-X_{a, b}\) is
distributed according to \(\pbeta(b,a)\). As the convex transform order is
invariant with respect to translations, Theorem~\ref{thm:mainifr} implies that
when \(a \leq b\) we have that \(-X_{a, b} \cleq X_{a, b}\). Since the convex
transform order orders only the underlying distribution the following definition
due to 
\textcite{VZ-79} is justified.
\begin{definition}[Positive/negative skew]\label{def:skew}
  Let \(P\) be a probability distribution and \(X \sim P\) a random variable
  with distribution \(P\). We say that \(P\) is \emph{positively skewed} if \(-X
  \cleq X\) and that \(P\) is \emph{negatively skewed} if \(X \cleq -X\).
\end{definition}
Thus, according to this definition, the Beta distributions have positive skew
when \(a \leq b\) and negative skew when \(a \geq b\).

As noted by 
\textcite{VZ-79} Definition~\ref{def:skew} provides an intuitive condition under
which inequalities between the mode, median, and mean hold. We give an
alternative proof of this fact. This alternative proof is based on the results
in the previous section and yields a similar inequality for the anti-mode.
\begin{theorem}\label{thm:mmm}
  Let \(P\) be a positively skewed distribution.

  If \(P\) is unimodal with mode \(m_{0}\), then there exists a median \(m_{1}\)
  of \(P\) such that \(m_{0} \leq m_{1}\).

  If \(P\) has finite mean \(m_{2}\), then there exists a median \(m_{1}\) of
  \(P\) such that \(m_{1} \leq m_{2}\).

  If \(P\) is uniantimodal with anti-mode \(m_{3}\), then there exists a median
  \(m_{1}\) of \(P\) such that \(m_{1} \leq m_{3}\).
\end{theorem}
\begin{proof}
  We prove only the first statement as the remaining ones are proved
  analogously. Let \(X\) be a random variable with distribution \(P\) and
  \(m_0\) the mode of \(P\). Then \(m_{1} = \sup\{ m \mid \mathbb{P}(X \leq m) \leq
  1/2 \}\) is a median of \(P\). Since \(P\) is positively skewed it follows by
  Corollary~\ref{corr:modejensen} that \(\mathbb{P}(X \leq m_0) \leq \mathbb{P}(-X \leq
  -m_0)\). Moreover, \(\mathbb{P}(-X \leq -m_0) = 1 - \mathbb{P}(X \leq m_0)\), so that
  \(\mathbb{P}(X \leq m_0) \leq 1/2\). Therefore \(m_{0} \leq m_{1}\). For the second
  statement apply Corollary~\ref{corr:mainexpectation} instead of
  Corollary~\ref{corr:modejensen}.
\end{proof}
Having a median lying between the mode and mean is usually
  called satisfying the \emph{mode-median-mean inequality}. Analogously we will
  say that a distribution satisfies the \emph{median-anti-mode} inequality if it
  has a median smaller than its anti-mode.

As noted above, when \(a \leq b\), the distribution \(\pbeta(a, b)\) is
positively skewed. The following slight generalisation of the known result
concerning the ordering of the mode, median, and mean of the Beta distribution
is now immediate (see for example \textcite{Runnenburg1978}).
\begin{corollary}
  If \(1 \leq a \leq b\) then \(\pbeta(a, b)\) satisfies the mode-median-mean
  inequality. If \(a \leq b \leq 1\) then \(\pbeta(a, b)\) satisfies the
  median-mean and median-anti-mode inequalities.
\end{corollary}

\section{Proofs}
\label{sec:proofs}

This section collects all the proofs related to establishing
Theorems~\ref{thm:betast},~\ref{thm:betaifra}~and~\ref{thm:mainifr}, stated in
Section~\ref{sec:main}. 

Most of the proofs rely on keeping track of sign changes of various functions.
Throughout \(S(x \in I \mapsto f(x)) = S(x \mapsto f(x)) = S(f(x)) = S(f) \in
\mathcal{S} = \{\snull, \sminus, \splus, \sminus\splus, \splus\sminus, \dotsc
\}\) denotes the sequence of signs of a function \(f \colon I \to \R\). Formal
definitions, notation, and standard results concerning sign patterns can be
found in later in Appendix~\ref{sec:overview}.

The following technical lemma summarises the basic strategy used throughout the
proofs of the main results in the upcoming sections.
\begin{lemma}\label{lemma:mainlemma}
  For \(a, b, a', b', c > 0\) and \(d < 1\) denote by \(F\) and \(G\) the
  distribution functions of \(\pbeta(a, b)\) and \(\pbeta(a', b')\) and
  \(\ell(x) = cx + d\). Then for \(I = \{ x \in [0, 1] \mid 0 < \ell(x) < 1 \} =
  (\max(0, -\frac{d}{c}), \min(1, \frac{1-d}{c}))\) one has
  \begin{align}
    S(x \in [0, 1] \mapsto F(x) - G(\ell(x)))
    & =
      S(x \in I \mapsto F(x) - G(\ell(x)))
      \label{eq:mainlemma1}
    \\ & \leq
         \sigma_{1} \cdot S(x \in I \mapsto p_{1}(x))
         \label{eq:mainlemma2}
    \\ & \leq
          \sigma_{1} \cdot S(x \in I \mapsto p_{2}(x))
         \label{eq:mainlemma3}
    \\ & \leq
         \sigma_{1} \cdot \sigma_{2} \cdot S(x \in I \mapsto p_{3}(x))
         \label{eq:mainlemma4}
    \\ & \leq
         \sigma_{1} \cdot \sigma_{2} \cdot S(x \in I \mapsto p_{4}(x)),
         \label{eq:mainlemma5}
  \end{align}
  where
  \begin{equation*}
    \sigma_{1} = \sign(-d),
    \qquad
    \sigma_{2} =
    \begin{cases}
      \sign(a' - a), & \text{if \(d = 0, a' \neq a\),} \\
      \sign(\mathrlap{1}\phantom{a'} - a), & \text{if \(d > 0, \mathrlap{a}\phantom{a'} \neq 1\),} \\
      \sign(a' - 1), & \text{if \(d < 0, a' \neq 1\),} \\
      \text{\(\snull\), \(\sminus\), or \(\splus\)}, & \text{otherwise,}\\
    \end{cases}
  \end{equation*}
  and
  \begin{align*}
    p_{1}(x)
    &=
      \frac{x^{a-1}(1-x)^{b-1}}{\betafun(a, b)}
      -
      \frac{\ell(x)^{a'-1}(1-\ell(x))^{b'-1}}{\betafun(a', b')},
    \\
    p_{2}(x)
    &=
      (a-1)\log(x) + (b-1)\log(1-x)
    \\
    &\quad
      -(a'-1)\log(\ell(x))-(b'-1)\log(1 - \ell(x)) + C,
    \\
    p_{3}(x)
    &=
            \frac{a-1}{x}
      -     \frac{b-1}{1-x}
      - \frac{c(a'-1)}{\ell(x)}
      +  \frac{c(b'-1)}{1-\ell(x)},
    \\
    p_{4}(x)
    &=
      c_{3}x^{3} + c_{2}x^{2} + c_{1}x + c_{0},
  \end{align*}
  for \(c_{3} = (a - a' + b - b')c^{2}\), \(c_{2} = -(a - a' + 1 - b')c^{2} - (a
  - a' + b - 1)c(1-d) - (b' - b + 1 - a)cd\), \(c_{1} = (a - a')c(1-d)
  - (a - b')cd - (a + b - 2)(1-d)d\), \(c_{0} = -(a-1)(d-1)d\), and \(C =
  \log\frac{\betafun(a', b')}{c\betafun(a, b)}\).
\end{lemma}
\begin{proof}
  Write \([0, 1] = J \cup I \cup J'\) where \(J = [0, \max(0, -d/c)]\) and \(J'
  = [\min(1, (1-d)/c), 1]\). Then \(S(x \in [0, 1] \mapsto F(x) - G(\ell(x))) =
  S(x \in J \mapsto F(x) - G(\ell(x))) \cdot S(x \in I \mapsto F(x) -
  G(\ell(x))) \cdot S(x \in J' \mapsto F(x) - G(\ell(x)))\). By construction the
  first and third terms are just a single sign that coincides with the first and
  final sign of \(S(x \in I \mapsto F(x) - G(\ell(x)))\) and can hence be
  dropped. This proves \eqref{eq:mainlemma1}.

  Assertion \eqref{eq:mainlemma2} is now immediate from
  Propositions~\ref{prop:signdiff}~and~\ref{prop:signdiffadd} and
  \eqref{eq:mainlemma3} follows by taking logarithms of both terms.

  Moreover, \eqref{eq:mainlemma4} follows by another application of
  Propositions~\ref{prop:signdiff}~and~\ref{prop:signdiffadd} and
  \eqref{eq:mainlemma5} follows by multiplication with
  \(x(1-x)\ell(x)(1-\ell(x))\) which is positive for \(x \in I\) by definition.
\end{proof}

\subsection{Stochastic dominance ordering}
\label{subsec:proofsst}

Before actually proving Theorem~\ref{thm:betast}, we shall prove that being ordered according to the stochastic dominance order
 is a \emph{necessary} condition for ordering compactly
supported distributions with respect to the star-shape transform or the convex
transform orders. Although the result concerning the stochastic dominance order is
well established, we present a proof using sign patterns.

A first result concerns a simple relation between the star-shaped transform
ordering and the stochastic dominance order.
\begin{proposition}\label{prop:cimpliesst}
  Let \(X \sim P\) and \(Y \sim Q\) be random variables with distributions \(P\)
  and \(Q\) supported on \([0, 1]\). Then \(X \IFRAleq Y\) implies \(Y \stleq
  X\).
\end{proposition}
\begin{proof}
  Let \(F\) and \(G\) be the distribution functions of \(X\) and \(Y\),
  respectively. As \(G^{-1}(F(x))/x\) is increasing, it follows that
  \(G^{-1}(F(x))/x \leq G^{-1}(F(1)) = 1\), thus \(G^{-1}(F(x)) \leq x\) and
  \(G(x) \geq F(x)\), meaning \(Y \stleq X\).
\end{proof}
Since the convex transform order implies the star-shape transform order, the
following is immediate.
\begin{corollary}\label{cor:cimpliesst}
  Let \(X \sim P\) and \(Y \sim Q\) be random variables with distributions \(P\)
  and \(Q\) supported on \([0, 1]\). Then \(X \cleq Y\) implies \(Y \stleq X\).
\end{corollary}
In the above statement the use of the unit interval is for notational
convenience. Using invariance under orientation-preserving affine
transformations the statement generalises to distributions on any bounded
interval.

Using the above we may now establish \emph{necessary}
conditions for one Beta distribution to be smaller than another according to
convex- or star-shaped transform orders. We do this by characterising when one
is smaller than the other according to stochastic dominance.

A proof of Theorem~\ref{thm:betast} can be found by elementary means, but since it illustrates well
the style of the upcoming proofs we formulate it in terms of an analysis of
sign patterns.
\begin{proof}[Proof of Theorem~\ref{thm:betast}]
  Let \(F\), \(G\), \(f\), and \(g\) be the distribution and density functions
  of \(\pbeta(a,b)\) and \(\pbeta(a', b')\). Denote \(H(x) = F(x) - G(x)\). We
  need to prove that \(S(H) = \sminus\) if and only if \(a \geq a'\) and \(b
  \leq b'\). We have
  \begin{equation*}
    \begin{split}
      S(H'(x))
      &=
      S\left(\frac{x^{a'-1}(1-x)^{b'-1}}{\betafun(a,b)}
        \left(
          x^{a-a'}(1-x)^{b-b'} - \frac{\betafun(a,b)}{\betafun(a',b')}
        \right)\right)
      \\&=
      S\left(
        x^{a-a'}(1-x)^{b-b'} - \frac{\betafun(a,b)}{\betafun(a',b')}
      \right).
    \end{split}
  \end{equation*}
  Since the case \(a = a'\) and \(b = b'\) is trivial, we may assume \(H\) is
  not constant \(0\) and so, since \(H(0) = H(1) = 0\), that neither \(S(H') =
  \splus\) nor \(S(H') = \sminus\).

  If \(a \geq a'\) and \(b \leq b'\), with at least one strict, we have \(S(H')
  \leq \sminus\splus\) since \(x^{a-a'}(1-x)^{b-b'}\) is increasing. Only
  \(S(H') = \sminus\splus\) is possible so
  Propositions~\ref{prop:signdiff}~and~\ref{prop:signdiffadd} imply \(\sminus
  \dotsb \sminus = S(H) \leq \sminus\splus\) with \(S(H) = \sminus\) the only
  option.

  Assume now that \(b > b'\) and \(a > a'\). Clearly \(S(H') = \sminus \dotsb
  \sminus\) and since \(x^{a-a'}(1-x)^{b-b'}\) is unimodal either \(S(H') =
  \sminus\) or \(S(H') = \sminus\splus\sminus\). Only \(S(H') =
  \sminus\splus\sminus\) is possible, so Proposition~\ref{prop:signdiffadd}
  implies that \(S(H) = \sminus \dotsb \splus \neq \sminus\).

  Using that \(\pbeta(a, b) \stleq \pbeta(a', b')\) if and only if \(\pbeta(b',
  a') \stleq \pbeta(a, b)\) and that \(\stleq\) is a partial order covers the
  remaining cases.
\end{proof}

\subsection{Star-shape ordering}
\label{subsec:proofsstar}

We now prove Theorem~\ref{thm:betaifra}, showing that, apart from reversing the order
direction, we find the same parameter characterisations as for the stochastic
dominance.

\begin{proof}[Proof of Theorem~\ref{thm:betaifra}]
  The necessity follows from Proposition~\ref{prop:cimpliesst} and
  Theorem~\ref{thm:betast}. As for the sufficiency, it is enough to prove the
  statement when \(a > a'\), \(b = b'\) and when \(a = a'\), \(b < b'\). The
  general statement then follows by transitivity since then \(\pbeta(a, b)
  \IFRAleq \pbeta(a, b') \IFRAleq \pbeta(a', b')\). Moreover, we may assume that
  either \(b \leq b' \leq 1\) or \(1 \leq b \leq b'\), since the remaining case,
  \(b \leq 1 \leq b'\), follows again by transitivity.

  Let \(F\) and \(G\) be the distribution functions of \(\pbeta(a, b)\) and
  \(\pbeta(a', b')\), respectively, with \(f\) and \(g\) the corresponding
  density functions as in \eqref{eq:dbetas}. By Proposition~\ref{prop:ifrasign}
  we need to prove that for every \(c \in \R\)
  \begin{equation}\label{eq:ifraeq}
    S(x \in [0, 1] \mapsto G^{-1}(F(x)) - cx) = S(F(x) - G(cx)) \leq \sminus\splus.
  \end{equation}
  As the assumptions on the parameters are the same as in
  Theorem~\ref{thm:betast}, it follows that \(G^{-1}(F(x)) \leq x\) and thus
  \eqref{eq:ifraeq} is satisfied when \(c \geq 1\). Moreover, both \(G^{-1}\)
  and \(F\) are increasing, so \eqref{eq:ifraeq} holds for \(c \leq 0\). It is
  therefore enough to consider \(c \in (0,1)\).

  The conclusion follows by analysing three different cases.
  \begin{description}
  \item[Case 1. \(b = b' \leq 1, \, a > a'\):]
    Using \eqref{eq:mainlemma2} from Lemma~\ref{lemma:mainlemma} with \(d = 0\) gives
    \begin{equation*}
      \begin{split}
        S(G^{-1}(F(x)) - cx)
        & \leq
        S\left(
          \frac{x^{a-1}(1-x)^{b-1}}{\betafun(a, b)}
          -
          \frac{c^{a'}x^{a'-1}(1-cx)^{b-1}}{\betafun(a', b)}
        \right)
        \\& =
        S
        \left(
          x^{a-a'}
          \left(
            \frac{1-cx}{1-x}
          \right)^{1-b}
          -
          c^{a'}\frac{\betafun(a, b)}{\betafun(a', b)}
        \right)
        \\& \leq
        \sminus\splus,
      \end{split}
    \end{equation*}
    since the last expression is increasing in \(x\) for \(x \in (0, 1)\), \(a \geq a'\), and \(b \leq 1\).

  \item[Case 2. \(b = b' \geq 1, \, a > a'\):]
    Applying Lemma~\ref{lemma:mainlemma} with \(d = 0\) we have \(c_{3} =
    (a-a')c^{2} > 0\), \(c_{2} = -(b-1)c(1-c) - (a-a')c(1+c)\), \(c_{1} =
    (a-a')c > 0\), \(c_{0} = 0\), \(\sigma_{1} = \snull\), and \(\sigma_{2} =
    \splus\), meaning \eqref{eq:mainlemma5} gives
    \begin{equation*}
      S(F(x)-G(cx))
      \leq
      \splus \cdot S(c_{3}x^{3} + c_{2}x^{2} + c_{1}x)
      =
      \splus \cdot S(c_{3}x^{2} + c_{2}x + c_{1}).
    \end{equation*}
    Since \(c_{3} > 0\) we have \(S(x \in [0, 1] \mapsto c_{3}x^{2} + c_{2}x +
    c_{1}) \leq \splus\sminus\splus\). But \(c_{1} > 0\) and \(c_{3} + c_{2} +
    c_{1} = -(b-1)c(1-c) < 0\) meaning we must have \(S(x \in [0, 1] \mapsto
    c_{2}x^{2} + c_{1}x + c_{0} ) = \splus\sminus\). Hence \(S(F(x) - G(cx))
    \leq \sminus\splus\sminus\). But since \(F(1) - G(c) = 1 - G(c) > 0\) we
    have \(S(F(x) - G(cx)) \leq \sminus\splus\).

  \item[Case 3. \(a = a', \, b < b'\):]
    Applying Lemma~\ref{lemma:mainlemma} with \(d = 0\) we have \(c_{3} =
    -(b'-b)c^{2} < 0\), \(c_{2} = (1-b)c - (1-b')c^{2}\), \(c_{1} = 0\), \(c_{0}
    = 0\), \(\sigma_{1} = \snull\), and \(\sigma_{2} \in \{\snull, \sminus,
    \splus\}\), meaning \eqref{eq:mainlemma5} gives
    \begin{equation*}
      S(F(x)-G(cx))
      \leq
      \sigma_{2} \cdot S(c_{3}x^{3} + c_{2}x^{2})
      =
      \sigma_{2} \cdot S(c_{3}x + c_{2})
      \leq
      \sigma_{2} \cdot \splus\sminus.
    \end{equation*}
    In any case \(S(F(x)-G(cx)) \leq \sminus\splus\sminus\) no matter the value
    of \(\sigma_{2}\). But \(F(1) - G(c) > 0\), so we must have \(S(F(x)-G(cx))
    \leq \sminus\splus\).
  \end{description}
  This concludes the proof.
\end{proof}
As will become apparent in the next section, this characterisation of the
star-shape transform ordering is an essential first step towards proving the
corresponding statement for the convex transform order.

\subsection{Convex transform ordering}
\label{subsec:proofsconvex}

To characterise how the Beta distributions are ordered according to the convex
transform order we will apply a strategy similar to the one used in previous
sections. According to Proposition~\ref{prop:ifrasign} we need to prove that for
\(a \geq a'\) and \(b \leq b'\) the distribution functions \(F\) and \(G\) of
\(\pbeta(a, b)\) and \(\pbeta(a', b')\), respectively, satisfy
\begin{equation}\label{eq:mainconvex}
  S(x \in [0, 1] \mapsto F(x) - G(\ell(x))) \leq \splus\sminus\splus,
\end{equation}
for every affine function \(\ell\).

First we need an auxiliary result, generalising Theorem~6.1~in~\textcite{AHO20},
which corresponds to taking \(x_0 = y_0 = 0 = \inf I\) in the statement below.
\begin{proposition}\label{prop:ifratrick}
  Let \(f \colon I \mapsto \R\) where \(I\) is an interval. If for some \(x_{0}
  \leq \inf I\) and \(y_0\) it holds that \(S(x \in I \mapsto f(x) - \ell(x))
  \leq \sminus\splus\) for all affine functions \(\ell\) such that \(\ell(x_{0})
  = y_0\) then \(S(x \in I \mapsto f(x) - \tilde{\ell}(x)) \leq \sminus\splus\)
  for all affine functions \(\tilde{\ell}\) such that \(\tilde{\ell}(x_{0}) \geq
  y_0\).

  The analogous conclusion holds considering the sign pattern \(\splus\sminus\)
  and taking \(x_{0} \geq \sup I\) satisfying \(\tilde{\ell}(x_{0}) \leq y_0\).
\end{proposition}
\begin{proof}[Proof (sketch)]
 The main idea is given graphically in Figure~\ref{fig:ifratrick} below where \(\ell\) is as in the statement and \(\tilde{\ell}\) is the line given by \(\tilde{\ell}(x_{0}) = y_{0}\) and \(\tilde{\ell}(x_{1}) = \ell(x_{1})\) for \(x_{1}\) the first time the graphs of \(f\) and \(\ell\) intersect, if it exists, and arbitrary otherwise.
\begin{figure}[h]
 \centering
  \includegraphics{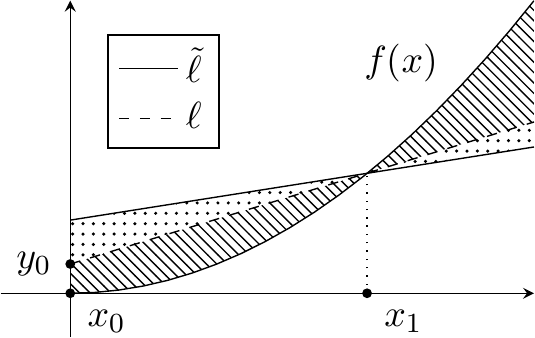}
\caption{Main idea of proof of Proposition~17}
\label{fig:ifratrick}
\end{figure}
\end{proof}

The above statement may be combined with the characterisation of how Beta
distributions are ordered according to the star-shaped transform order that was
established in the previous section. Doing so allows us to immediately take care
of a number of affine \(\ell\) in \eqref{eq:mainconvex}.

\begin{corollary}\label{corr:triviallines}
  Let \(F\) and \(G\) be the distribution functions of \(\pbeta(a, b)\) and
  \(\pbeta(a', b')\), respectively, and assume that \(a \geq a'\) and \(b \leq
  b'\). If \(\ell\) is decreasing or satisfies either \(\ell(0) \geq 0\) or
  \(\ell(1) \not \in (0, 1)\) then \eqref{eq:mainconvex} is satisfied.
\end{corollary}
\begin{proof}
  We have that \eqref{eq:mainconvex} holds when \(\ell\) is non-increasing, since otherwise \(F(x) - G(\ell(x))\) is non-decreasing and so \(S(F(x) - G(\ell(x))) \leq \sminus\splus\). Therefore assume \(\ell\) is increasing and consider three different cases.
  \begin{description}
  \item[Case 1. \(\ell(0) \geq 0\):] According to Theorem~\ref{thm:betaifra} and
    Proposition~\ref{prop:ifrasign}, we have that, for any \(c \in \R\),
    \(S(F(x) - G(cx)) \leq \sminus\splus\). Taking into account
    Proposition~\ref{prop:ifratrick}, this implies \(S(F(x) - G(\ell(x))) \leq
    \sminus\splus \leq \splus\sminus\splus\).

  \item[Case 2. \(\ell(1) \leq 0\):] In this case \(\ell\) is always negative,
    and the result is immediate.

  \item[Case 3. \(\ell(1) \geq 1\):] It is enough to prove that \(S(F(x) -
    G(c(x-1)+1)) \leq \splus\sminus\) for every \(c \in \R\). Indeed, once this
    proved, the conclusion follows using Proposition~\ref{prop:ifratrick} again.

    Note that \(F_-(x) = 1-F(1-x)\) is the distribution function of \(\pbeta(b,
    a)\) and \(G_-(x) = 1-G(1-x)\) is the distribution function of \(\pbeta(b',
    a')\). The characterisation of star-shape transform order proved in
    Theorem~\ref{thm:betaifra} together with Proposition~\ref{prop:ifrasign},
    means that \(S(1-G(1-x)-1+F(1-c'x)) \leq \sminus\splus\), for every \(c' \in
    \R\). For any \(c \in \R\) we may apply this to \(c' = 1/c\), which gives
    \begin{equation*}
      \begin{split}
        S(F(x) - G(c(x-1)+1)) & = S(1 - G(c(x-1)+1) - 1 + F(x)) \\
        & = \srev S(1 - G(1-x) - 1 + F(1-x/c)) \\
        & \leq \srev(\sminus\splus) = \splus\sminus.
      \end{split}
    \end{equation*}
  \end{description}
\end{proof}

The proof of Theorem~\ref{thm:mainifr}, establishing the convex transform
ordering within the Beta family is achieved through the analysis of several
partial cases. For improved readability we will be presenting these in several
lemmas.

\begin{lemma}\label{lemma:bfixoneabig}
  Let \(X \sim \pbeta(a, b)\) and \(Y \sim \pbeta(1, b)\) for some \(a \geq 1\) and
  \(b > 0\). Then \(X \cleq Y\).
\end{lemma}
\begin{proof}
  Let \(F\) and \(G\) be the distribution functions of the \(\pbeta(a, b)\) and
  \(\pbeta(1, b)\) distributions, respectively, and \(f\) and \(g\) their
  densities. Taking into account Proposition~\ref{prop:ifrasign} and
  Corollary~\ref{corr:triviallines}, we need to show that, for every increasing affine
  function \(\ell(x) = cx + d\) satisfying \(\ell(0) = d < 0\) and \(\ell(1) =
  c+d \in (0, 1)\) one has that \eqref{eq:mainconvex} is satisfied. We need to
  separate the arguments into three cases.
  \begin{description}
  \item[Case 1. \(b = 1\):] In this case the statement follows directly from the
    convexity of \(F(x) = x^{a}\) and that \(G(x) = x\).

  \item[Case 2. \(b \in (0,1)\):]
    Applying Lemma~\ref{lemma:mainlemma} we have \(I = (-d/c, 1)\), \(\sigma_{1}
    = \splus\), and \(\sigma_{2} \in \{\snull, \sminus, \splus\}\), meaning
    \eqref{eq:mainlemma4} gives
    \begin{equation*}
      S(F(x)-G(\ell(x)))
      \leq
      \splus \cdot \sigma_{2} \cdot
      S\left(\frac{a-1}{x}-\frac{b-1}{1-x}+\frac{c(b-1)}{1-\ell(x)}\right).
    \end{equation*}
    But since for \(x \in I\)
    \begin{equation*}
      \frac{a-1}{x}-\frac{b-1}{1-x}+\frac{c(b-1)}{1-
        \ell(x)}
      =
      \frac{a-1}{x} +
      \frac{(1-b)(1-(c+d))}{(1-x)(1- \ell(x))} > 0,
    \end{equation*}
    we have \(S(F(x) - G(\ell(x))) \leq \splus \cdot \sigma_{2} \cdot \splus
    \leq \splus\sminus\splus\) no matter the value of \(\sigma_{2}\).

  \item[Case 3. \(b > 1\):]
    Applying Lemma~\ref{lemma:mainlemma} we have \(\sigma_{1} = \splus\) and
    \(\sigma_{2} \in \{\snull, \sminus, \splus\}\), meaning
    \eqref{eq:mainlemma5} gives
    \begin{equation*}
      \begin{split}
        S(F(x)-G(\ell(x)))
        &\leq
        \splus \cdot \sigma_{2} \cdot S(c_{3}x^{3} + c_{2}x^{2} + c_{1}x + c_{0})
        \\&=
        \splus \cdot \sigma_{2} \cdot S(\ell(x)(c'_{2}x^{2} + c'_{1}x + c'_{0}))
        \\&=
        \splus \cdot \sigma_{2} \cdot S(c'_{2}x^{2} + c'_{1}x + c'_{0})
      \end{split}
    \end{equation*}
    where \(c'_{2}=(a-1)c\), \(c'_{1} = -(a-b)c - (a+b-2)(1-d)\), and \(c'_{0} =
    (a-1)(1-d)\). Since \(c'_{2} > 0\) we have \(S(c'_{2}x^{2} + c'_{1}x +
    c'_{0}) \leq \splus\sminus\splus\). On the other hand,
    \(c'_{2}+c'_{1}+c'_{0} = -(b-1)(1-(c+d)) < 0\), hence
    \(S(c'_{2}x^{2}+c'_{1}x+c'_{0}) \leq \splus\sminus\). Combining these
    inequalities yields
    \begin{equation*}
      S(F(x) - G(\ell(x))) \leq\splus \cdot \sigma_{2} \cdot \splus\sminus \leq \splus\sminus\splus\sminus.
    \end{equation*}
    Finally, as \(F(1) - G(\ell(1)) > 0\), it follows that \(S(F(x) -
    G(\ell(x))) \leq \splus\sminus\splus\).
  \end{description}
  So, taking into account Proposition~\ref{prop:ifrasign}, the proof is concluded.
\end{proof}

The second lemma is similar but covers the case where \(a \geq 1\).
\begin{lemma}\label{lemma:bfixoneasmall}
  Let \(X \sim \pbeta(1, b)\) and \(Y \sim \pbeta(a, b)\) with \(a \leq 1\) and
  \(b > 0\). Then \(X \cleq Y\).
\end{lemma}
\begin{proof}
  Let \(F\) and \(G\) represent the distribution functions of \(\pbeta(1, b)\)
  and \(\pbeta(a, b)\), respectively. Note that the meaning of the symbols \(F\)
  and \(G\) are interchanged relative to their use in the proof of
  Lemma~\ref{lemma:bfixoneabig}. Taking into account
  Proposition~\ref{prop:ifrasign} and Corollary~\ref{corr:triviallines}, we need
  to show that \eqref{eq:mainconvex} holds for \(\ell(x) = cx + d\) such that
  \(\ell(0) = d < 0\) and \(\ell(1) = c + d \in (0, 1)\). This is equivalent to
  \(S(G(x) - F(\ell^{-1}(x))) \leq \sminus\splus\sminus\), where \(\ell^{-1}(x)
  = (x-d)/c\) satisfies \(\ell^{-1}(0) \in (0,1)\) and \(\ell^{-1}(1) > 1\).

  Reversing the roles of \(F\) and \(G\) the proof is now analogous to that of
  Lemma~\ref{lemma:bfixoneabig} except that \(a < 1\) and we wish to establish
  \(S(x \in I \mapsto G(x) - F(\ell^{\ast}(x))) \leq \sminus\splus\sminus\) for
  \(\ell^{\ast}(x) = c^{\ast}x+d^{\ast}\) with \(\ell^{\ast}(0) = d^{\ast} \in
  (0,1)\) and \(\ell^{\ast}(1) = c^{\ast} + d^{\ast} > 1\) on the interval \(I =
  (0, (1-d^{\ast})/c^{\ast})\).
\end{proof}

Comparing the distributions for more general pairs of parameters \(a\) and
\(a'\) requires separate analyses depending on whether \(b > 1\) or \(b \in (0,
1)\).
\begin{lemma}\label{lemma:bfixbiganotacross1}
  Let \(X\sim\pbeta(a, b)\) and \(Y\sim\pbeta(a', b)\) with \(a > a'\) and \(b >
  1\). Then \(X \cleq Y\).
\end{lemma}
\begin{proof}
  Let \(F\) and \(G\) be the distribution functions of \(\pbeta(a, b)\) and
  \(\pbeta(a', b)\). By Proposition~\ref{prop:ifrasign} and
  Corollary~\ref{corr:triviallines}, it is enough to prove that
  \eqref{eq:mainconvex} holds when \(\ell(x) = cx + d\) is such that \(\ell(0) =
  d < 0\), \(\ell(1) = c + d \in (0, 1)\).

  Applying Lemma~\ref{lemma:mainlemma} we have \(c_{3} = (a - a')c^{2}\),
  \(\sigma_{1} = \splus\), and \(\sigma_{2} \in \{\snull, \sminus, \splus\}\),
  meaning \eqref{eq:mainlemma5} gives
  \begin{equation*}
    S(F(x)-G(\ell(x)))
    \leq
    \splus \cdot \sigma_{2} \cdot S(c_{3}x^{3} + c_{2}x^{2} + c_{1}x^{1} + c_{0}).
  \end{equation*}
  Since \(c_{3} > 0\) we have \(S(c_{3}x^{3} + c_{2}x^{2} + c_{1}x^{1} + c_{0})
  \leq \sminus\splus\sminus\splus\). But \(c_{3} + c_{2} + c_{1} + c_{0} =
  -(b-1)(c+d)(1-(c+d)) < 0\) so \(S(c_{3}x^{3} + c_{2}x^{2} + c_{1}x^{1} + c_{0})
  \leq \sminus\splus\sminus\).

  Combining the sign pattern inequalities, we have derived that
  \begin{equation*}
    S(x \in I \mapsto F(x) - G(\ell(x)))
    \leq
    \splus \cdot \sigma_{2} \cdot \sminus\splus\sminus
    =
    \splus\sminus\splus\sminus,
  \end{equation*}
  regardless of the value of \(\sigma_{2}\). Finally \(F(1) - G(\ell(1)) > 0\)
  so we conclude that \(S(x \in I \mapsto F(x) - G(\ell(x))) \leq
  \splus\sminus\splus\), hence proving the result.
\end{proof}

\begin{lemma}\label{lemma:bfixsmallbothabigorsmall}
  Let \(X \sim \pbeta(a, b)\) and \(Y \sim \pbeta(a', b)\) with \(0 < b \leq 1\)
  and either \(1 > a > a' > 0\) or \(a > a' > 1\). Then \(X \cleq Y\).
\end{lemma}
\begin{proof}
  We may assume, without loss of generality, that \(a-a' < 1\). Indeed, if
  \(a-a' \geq 1\), one may choose for sufficiently large \(N\) a sequence
  \(a_{0} = a, a_{1}, \dotsc, a_{N} = a'\) such that \(a_{i_{1}}-a_{i} < 1\) for
  all \(i = 1, \dotsc, N\) and apply transitivity to conclude \(\pbeta(a_{0}, b)
  \cleq \dotsb \cleq \pbeta(a_{N}, b)\). Let \(F\) and \(G\) be the distribution
  functions of \(\pbeta(a, b)\) and \(\pbeta(a', b)\), respectively. Based on
  Proposition~\ref{prop:ifrasign} and Corollary~\ref{corr:triviallines}, it is
  enough to prove that \eqref{eq:mainconvex} holds for every \(\ell(x) = cx+d\)
  such that \(\ell(0) = d < 0\) and \(\ell(1) = c+d \in (0, 1)\). Using
  Lemma~\ref{lemma:mainlemma} we have for \(I = (-d/c, 1)\) that
  \begin{equation}\label{eq:maingeneralsameb}
    \begin{split}
      &S(x \in [0,1] \mapsto F(x) - G(\ell(x)))
      \\&\quad\leq
      \splus \cdot S
      \left(
        x \in I \mapsto
        \frac{x^{a-1}(1-x)^{b-1}}{\betafun(a,b)}
        -
        \frac{\ell(x)^{a'-1}(1-\ell(x))^{b-1}}{\betafun(a',b)}
      \right)
      \\&\quad=
      \splus \cdot S
      \left(
        x \in I \mapsto
        \frac{x^{a-1}}{\ell(x)^{a'-1}}
        -
        \frac{c\betafun(a,b)}{\betafun(a',b')}
        \Bigl(\frac{1-\ell(x)}{1 - x}\Bigr)^{b-1}
      \right).
    \end{split}
  \end{equation}
  For convenience define \(C = c\betafun(a,b)/\betafun(a',b') > 0\), \(q_{1}(x)
  = x^{a-1}/\ell(x)^{a'-1}\), \(q_{2}(x) = (1-x)/(1-\ell(x))\), and \(q(x) =
  q_{1}(x) - Cq_{2}(x)^{1-b}\). Restricting to \(I\) we have that \(q_{2}\) is
  decreasing and concave. Hence, as \(b \leq 1\), it follows that \(x \in I
  \mapsto -Cq_{2}(x)^{1-b}\) is non-decreasing and convex.

  A simple computation yields \(q'_{1}(x) = ((a-a')cx +
  (a-1)d)/(x^{2-a}\ell(x)^{a'})\) which has unique root at \(x_{0} =
  -(a-1)d/((a-a')c)\). Letting \(c_{2}^{\ast} = (a-a')(a-a'-1)c^{2}\),
  \(c_{1}^{\ast} = 2(a-a'-1)(a-1)cd\), \(c_{0}^{\ast} = (a-1)(a-2)d^{2}\), and
  \(p(x) = c_{2}^{\ast}x^{2} + c_{1}^{\ast}x + c_{0}^{\ast}\) it follows that \(q_{1}''(x) =
  p(x)/(x^{3-a}\ell(x)^{a'+1})\).
  \begin{description}
  \item[Case 1. \(1 > a > a'\):]
    In this case \((a-a')cx + (a-1)d > 0\) for \(x > 0\), which implies that
    \(q_{1}\) is increasing on \(I\). Therefore \(x \in I \mapsto
    q_{1}(x)-Cq_{2}(x)^{1-b}\) is increasing, meaning \(S(x \in I \mapsto
    q_{1}(x)-Cq_{2}(x)^{1-b}) \leq \sminus\splus\). Plugged into
    \eqref{eq:maingeneralsameb} this gives \(S(F(x) - G(\ell(x))) \leq
    \splus\sminus\splus\).

  \item[Case 2. \(a > a' > 1\):]
    A direct verification shows that \(x_{0} \in I\) so that \(I_{1} = (-d/c,
    x_0]\) and \(I_{2} = (x_0,1)\) are well defined and non-empty. Since \(I =
    I_{1} \cup I_{2}\) and \(I_{1} < I_{2}\)
    \begin{equation*}
      S(x \in I \mapsto q(x))
      =
      S(x \in I_{1} \mapsto q(x)) \cdot S(x \in I_{2} \mapsto q(x)).
    \end{equation*}

    \begin{description}
    \item[Sign pattern in \(I_{1}\):]
      As \(c_{2}^{\ast} = (a-a')(a-a'-1)c^{2} < 0\) it follows that \(S(x \in I
      \mapsto q_{1}''(x)) = S(x \in I \mapsto p(x)) \leq \sminus\splus\sminus\).
      But \(p(-d/c) = (a'-1)a'd^{2} > 0\) and \(p(x_0) = (a-1)(a'-1)d^{2}/(a-a')
      > 0\) so \(S(x \in I \mapsto p(x)) = \splus\).

      This implies that \(q_{1}\) is convex in \(I_{1}\). As we have proved the
      convexity of \(-Cq_{2}(x)^{1-b}\) in \(I\), it follows that \(q(x) =
      q_{1}(x)-Cq_{2}(x)^{1-b}\) is convex in \(I_{1}\). According to
      Proposition~\ref{prop:convexconvexsigns} it follows that
      \begin{equation*}
        S(x \in I_{1} \mapsto q(x)) \leq \splus\sminus\splus.
      \end{equation*}
    \item[Sign pattern in \(I_{2}\):] Noting that \(S(x \in I \mapsto q_{1}'(x))
      = S((a-a')cx + (a-1)d) \leq \sminus\splus\) and \(q_{1}'(x_0) = 0\), it
      follows that \(q_{1}'\) is positive in \(I_{2}\). Thus \(q_{1}\) is
      increasing in \(I_{2}\). We have proved above that \(x \in I \mapsto
      -Cq_{2}(x)^{1-b}\) is increasing, so \(q(x)\) is increasing in the
      interval \(I_{2}\). Therefore
      \begin{equation*}
        S(x \in I_{2} \mapsto q(x)) \leq \sminus\splus.
      \end{equation*}
    \end{description}

    If \(q(x_{0}) < 0\) then \(S(x \in I_{1} \mapsto q(x)) \leq \splus\sminus\).
    If \(q(x_{0}) \geq 0\) then \(S(x \in I_{2} \mapsto q(x)) = \splus\) since
    \(q\) is increasing on \(I_{2}\). In either case \(S(x \in I_{1} \mapsto
    q(x)) \cdot S(x \in I_{2} \mapsto q(x)) \leq \splus\sminus\splus\).

    Putting the above into \eqref{eq:maingeneralsameb} we have
    \begin{equation*}
      S(x \in I \mapsto F(x) - G(\ell(x)))
      \leq
      \splus \cdot \splus\sminus\splus = \splus\sminus\splus
    \end{equation*}
    as required.
  \end{description}
\end{proof}
We now state, without proof, a straightforward result, helpful for the
conclusion of the final characterisation within the \(\pbeta\) family.
\begin{proposition}\label{prop:mirrorifr}
  Let \(X \sim P\) and \(Y \sim Q\) be random variables with some distributions
  \(P\) and \(Q\), then \(X \cleq Y\) if and only if \(1-Y \cleq 1-X\).
\end{proposition}

We now have all the necessary ingredients to prove the main theorem.

\begin{proof}[Proof of Theorem~\ref{thm:mainifr}]
  The necessity is a direct consequence of Theorem~\ref{thm:betast}. The
  sufficiency follows from Lemmas~\ref{lemma:bfixoneabig},
  \ref{lemma:bfixoneasmall}, \ref{lemma:bfixbiganotacross1},
  \ref{lemma:bfixsmallbothabigorsmall}, and the transitivity of the convex
  transform order. First note that we obtain
  \begin{equation}
    \label{eq:beta1}
    \pbeta(a, b) \cleq \pbeta(a', b),
  \end{equation}
  when \(a = a'\) (trivial), \(b > 1\) (use
  Lemma~\ref{lemma:bfixbiganotacross1}), \(b \leq 1\) and either \(1 > a > a'\)
  or \(a > a' > 1\) (use Lemma~\ref{lemma:bfixsmallbothabigorsmall}). The order
  relation (\ref{eq:beta1}) also holds if \(b \leq 1\) and \(a > 1 > a'\) by
  combining Lemmas~\ref{lemma:bfixoneabig}~and~\ref{lemma:bfixoneasmall}, since
  then \(\pbeta(a, b) \cleq \pbeta(1, b) \cleq \pbeta(a', b)\). Using this and
  Proposition~\ref{prop:mirrorifr} we also have \(\pbeta(a', b) \cleq \pbeta(a',
  b')\), concluding the proof.
\end{proof}

\section*{Acknowledgments}

We would like to thank the anonymous reviewers for their detailed remarks and extensive references.

\printbibliography{}

\appendix

\section{An algebra for sign variation}
\label{sec:overview}

The main tool of all proofs concerning the ordering within the \(\pbeta\) family
is the study of sign patterns of functions. While such techniques have a long
tradition in probability theory, 
for our purposes it turns out to be computationally convenient to give a
presentation slightly more algebraic as compared to what appears to be the
convention, using a suitable monoid (see, for example, 
\textcite{Jac85}).
\begin{definition}
  Let \((\mathcal{S}, \cdot ) = \langle \splus, \sminus \mid \splus \cdot \splus
  = \splus, \sminus \cdot \sminus = \sminus\rangle\) be the monoid generated by
  two idempotent elements \(\splus\) and \(\sminus\) and with unit \(\snull\).
\end{definition}
We shall call elements of \(\mathcal{S}\) \emph{sign patterns}. When unambiguous
we will denote products \(\sigma \cdot \sigma'\) by simply juxtaposing the
factors as in \(\sigma\sigma'\), so that \(\mathcal{S} = \{\snull, \splus,
\sminus, \splus\sminus, \sminus\splus, \splus\sminus\splus, \dotsc\}\).

For any \(\sigma = \sigma_{1} \dotsb \sigma_{n} \in \mathcal{S}\) where
\(\sigma_{1}, \dotsc, \sigma_{n} \in \{\splus, \sminus\}\) let \(\srev \sigma =
\sigma_{n} \dotsb \sigma_{1}\) be the sign pattern given by reversing the order
of signs and let
\begin{equation*}
  \label{eq:signflip}
  \overline{\sigma} = \overline{\sigma_{0}} \cdot \dotsb \cdot \overline{\sigma_{n}},
  \qquad \text{where} \qquad \overline{\sigma_{i}} =
  \begin{cases}
    \splus & \sigma_{i} = \sminus,
    \\
    \sminus & \sigma_{i} = \splus,
  \end{cases}
\end{equation*}
denote the sign pattern given by flipping the signs. Note in particular that
\(\overline{\snull} = \snull\). These operations are  well defined on the free monoid generated by \(\splus\) and \(\sminus\). Since \(\srev(\snull) = \snull\), \(\overline{\snull} = \snull\), \(\srev(\sigma_{1} \dotsb \sigma_{n} \cdot \sigma_{n+1}) = \sigma_{n+1} \cdot \srev(\sigma_{1} \dotsb \sigma_{n})\),
and \(
\overline{\sigma_{1} \dotsb \sigma_{n} \cdot \sigma_{n+1}}
= 
\overline{\sigma_{1} \dotsb \sigma_{n}} \cdot \overline{\sigma_{n+1}}
\)
it follows by a simple induction argument that they are well defined also as operations on \(\mathcal{S}\).

Sign patterns have a natural order structure.
\begin{definition}
  Given \(\sigma, \sigma' \in \mathcal{S}\) we say that \(\sigma \leq \sigma'\)
  if \(\sigma' = \pi \cdot \sigma \cdot \pi'\) for some \(\pi, \pi' \in
  \mathcal{S}\).
\end{definition}
Intuitively \(\sigma \leq \sigma'\) says that \(\sigma\) may be written as a substring of \(\sigma'\).

\begin{proposition}
  \((\mathcal{S}, \cdot, \leq)\) is a partially ordered monoid in the sense that
  \((\mathcal{S}, \leq)\) is a partially ordered set and if \(\sigma, \sigma'
  \in \mathcal{S}\) are such that \(\sigma \leq \sigma'\) then for any \(\pi,
  \pi' \in \mathcal{S}\) one has \(\pi \cdot \sigma \cdot \pi' \leq \pi \cdot
  \sigma' \cdot \pi'\).
\end{proposition}
We can now describe the sign variations of a function in terms of the simple
sign function.
\begin{definition}[Sign function]
  The sign function \(\sign \colon \R \to \mathcal{S}\) is defined by
  \(\sign(x) = \splus\) if \(x > 0\), \(\sign(x) = \snull\) if \(x = 0\), and
  \(\sign(x) = \sminus\) if \(x < 0\).
\end{definition}

\begin{definition}[Sign patterns and finite sign variation]
  Given \(I \subseteq \R\), we say that a function \(f \colon I \mapsto \R\) is
  of \emph{finite sign variation} if the set
  \begin{equation*}
    \left\{\sign(f(x_{1})) \cdot \sign(f(x_{2})) \cdot \dotsb \cdot \sign(f(x_{n})) \mid n \in \N, x_{1} \leq \dotsb \leq x_{n} \in I \right\}
  \end{equation*}
  has a (unique) maximal element in \(\mathcal{S}\). This maximal element is
  then denoted by \(S(x \in I \mapsto f(x))\) and called the \emph{sign pattern}
  of \(f\).
\end{definition}
When unambiguous, we will abbreviate \(S(x \in I \mapsto f(x)) = S(x \mapsto
f(x)) = S(f(x)) = S(f)\) and write for readability \(\overline{S}(f) =
\overline{S(f)}\).

The proposition below gives some standard rules of calculation for sign patterns
which are straightforward to prove and used without explicit mention throughout
the proofs.
\begin{proposition}
  Let \(I \subset \R\) and \(f,g \colon I \to \R\) be such that \(f\) and
  \(f-g\) are of finite sign variation.
  \begin{enumerate}
  \item For any \(J \subset I\) one has \(S(x \in J \mapsto f(x)) \leq S(x \in I
    \mapsto f(x))\).

  \item For any \(J \leq K\) such that \(I = J \cup K\) one has \(S(x \in I
    \mapsto f(x)) = S(x \in J \mapsto f(x)) \cdot S(x \in K \mapsto f(x))\).

  \item For any positive \(h \colon I \to \R\) one has \(S(f(x)) =
    S(f(x)h(x))\).

  \item For \(J \subset \R\) and \(\eta \colon J \to I\) increasing (or
    decreasing) one has \(S(x \in I \mapsto f(x)) = S(x \in J \mapsto
    f(\eta(x)))\) (respectively \(= \srev S(x \in J \mapsto f(\eta(x)))\)).

  \item For \(J \subset f(I) \cup g(I)\) and \(\eta \colon J \to \R\) increasing
    (or decreasing) one has \(S(f(x) - g(x)) = S(\eta(f(x)) - \eta(g(x)))\)
    (respectively \(= \overline{S}(\eta(f(x)) - \eta(g(x)))\)).
  \end{enumerate}
\end{proposition}

Sign patterns provide a useful tool for establishing convexity or
star-shapedness of functions (see for example Lemma~11 and Theorem~20 in
\textcite{AO19}).
\begin{proposition}\label{prop:convexconvexsigns}
  A continuous function \(f\) is convex (respectively, star-shaped) if and only
  if \(S(f(x) - \ell(x)) \leq \splus\sminus\splus\) (respectively, \(S(f(x)-
  \ell(x)) \leq \sminus\splus\)), for all affine functions \(\ell\)
  (respectively, for all affine functions \(\ell\) vanishing at 0).
\end{proposition}

Applied to the convex (\textsc{ifr}) and star-shape transform (\textsc{ifra})
orders, these characterisations translate into the following equivalent
conditions for being ordered.
\begin{proposition}\label{prop:ifrasign}
  Let \(X\) and \(Y\) be random variables with distributions given by
  distribution functions \(F\) and \(G\), respectively. Then \(X \cleq Y\)
  (respectively \(X \IFRAleq Y\)) if and only if \(S(F(x) - G(\ell(x))) \leq
  \splus\sminus\splus\) (resp., \(S(F(x) - G(\ell(x))) \leq \sminus\splus\)) for
  every affine function \(\ell\) (resp., for every affine function \(\ell\)
  vanishing at 0).
\end{proposition}

The following slight generalisation of a well-known relationship between the
sign pattern of a differentiable function and the sign pattern of its derivative
is also used throughout our proofs.
\begin{proposition}\label{prop:signdiff}
  Let \(f \colon I \mapsto \R\) be continuously differentiable with finite sign
  pattern \(S(x \in I \to f(x)) = \sigma \dotsb \), then \(S(x \in I \to f(x))
  \leq \sigma \cdot S(x \in I \to f'(x))\).
\end{proposition}
\begin{proof}
  Let \(S(x \in I \to f(x)) = \sigma_{0} \sigma_{1} \dotsb \sigma_{n}\).
  Therefore there exists a sequence \(x_{0} < x_{1} < \dotsb < x_{n}\) with
  \(\sign(f(x_{i})) = \sigma_{i}\). By the mean value theorem there exist
  \(y_{1}, \dotsc, y_{n}\) such that \(f'(y_{i}) =
  (f(x_{i})-f(x_{i-1}))/(x_{i}-x_{i-1})\). Since, in particular,
  \(\sign(f'(y_{i})) = \sigma_{i}\), we have that \(\sigma_{1} \dotsb \sigma_{n}
  \leq S(x \in I \to f'(x))\).
\end{proof}

If in the statement of Proposition~\ref{prop:signdiff} the initial sign of \(S(x
\in I \to f'(x))\) is the same as \(\sigma\) the inequality becomes \(S(x \in I
\to f(x)) \leq S(x \in I \to f'(x))\). This becomes particularly useful in
combination with the following, elementary, proposition.

\begin{proposition}\label{prop:signdiffadd}
  For \(b > a\) let \(f \colon [a, b] \mapsto \R\) be a continuously
  differentiable function with finite sign patterns \(S(x \in I \to f(x)) =
  \sigma \dotsb \sigma'\) and \(S(x \in I \to f'(x))) = \tau \dotsb \tau'\). If
  \(f(a) = 0\) then \(\sigma = \tau\) and if \(f(b) = 0\) then \(\sigma' =
  \overline{\tau}'\).

  The interval \([a, b]\) may be replaced by \((a, b]\), \([a, b)\) or \((a,
  b)\) if the conditions \(f(a) = 0\) and \(f(b) = 0\) are replaced by \(\lim_{x
    \to a+}f(x) = 0\) or \(\lim_{x \to b-}f(x) = 0\), as appropriate.
\end{proposition}

\end{document}